\documentclass{amsart}

\usepackage{amsmath}
\usepackage{amssymb}
\usepackage{amsfonts}
\usepackage{amsthm}
\usepackage{mathrsfs}
\usepackage[numbers]{natbib}

\usepackage{comment}

\usepackage[colorlinks=true,urlcolor=blue, citecolor=red,linkcolor=blue,linktocpage,pdfpagelabels, bookmarksnumbered,bookmarksopen]{hyperref}

\newtheorem{theorem}{Theorem}
\newtheorem*{theorem*}{Theorem}
\newtheorem{lemma}[theorem]{Lemma}
\newtheorem{corollary}[theorem]{Corollary}

\theoremstyle{definition}
\newtheorem{definition}[theorem]{Definition}

\newcommand{\R}{\mathbb{R}}
\newcommand{\C}{\mathbb{C}}

\newcommand{\N}{\mathbb{N}}

\numberwithin{theorem}{section} 
\numberwithin{equation}{section}

\title[]{\protect{Phragm\'{e}n-Lindel\"{o}f Principles and Julia Limiting Directions of Quasiregular Mappings}}

\AtBeginDocument{
  \setlength\abovedisplayskip{5pt}
  \setlength\belowdisplayskip{5pt}}

\begin{document}
\author{Alastair N. Fletcher}
\address{Department of Mathematical Sciences, Northern Illinois University, DeKalb, IL 60115-2888. USA}
\email{afletcher@niu.edu}

\author{Julie M. Steranka}
\address{Department of Mathematical Sciences, Northern Illinois University, DeKalb, IL 60115-2888. USA}
\email{jsteranka1@niu.edu}

\begin{abstract}
We show that the set of Julia limiting directions of a transcendental-type $K$-quasiregular mapping $f:\R^n\to \R^n$ must contain a component of a certain size, depending on the dimension $n$, the maximal dilatation $K$, and the order of growth of $f$. In particular, we show that if the order of growth is small enough, then every direction is a Julia limiting direction. We also show that if every component of the set of Julia limiting directions is a point, then $f$ has infinite order. The main tool in proving these results is a new version of a Phragm\'en-Lindel\"of principle for sub-$F$-extremals in sectors, where we allow for boundary growth of the form $O( \log |x| )$ instead of the previously considered $O(1)$ bound.
\end{abstract}
\maketitle

\section{Introduction}
\subsection{Julia limiting directions}

There is an increasingly well-developed iteration theory of transcendental-type quasiregular mappings in space. This was first put on a firm theoretical foundation by Bergweiler and Nicks \cite{BN} who showed that there is a definition of the Julia set that appropriately generalizes the usual notion of the Julia set in complex dynamics. The big difference in the quasiregular context is that there is no longer any guarantee of local normality of the family of iterates, even outside the Julia set. For this reason, the complement of the Julia set $J(f)$ of a transcendental-type quasiregular mapping $f$ is called the quasi-Fatou set $QF(f)$ instead of just the Fatou set.

The main focus of this paper is on the set of Julia limiting directions $L(f)$ of a transcendental-type quasiregular mapping $f$. This is sometimes called the radial distribution of $J(f)$. We briefly recall the historical development of $L(f)$. 
The set of Julia limiting directions was first studied by Qiao \cite{Qiao} for transcendental entire functions in the plane. In this setting $L(f) \subset S^1$. One of the main results of this paper \cite[Theorem 1]{Qiao} shows that if the lower order $\lambda_f$ of $f$ is finite, then there exists a closed interval $I\subset S^1$ such that $I\subset L(f)$ and there is a lower bound on the measure of this interval, that is, $m(I)\geq \frac{\pi}{\max(1/2,\lambda_f)}$. In particular, if $\lambda_f \leq 1/2$ then $L(f) = S^1$.
A number of papers have studied $L(f)$ for meromorphic functions, including \cite{QW,WY,W, ZWH}, as well as for solutions of differential equations, see for example \cite{WC}.

In higher dimensions, if $f:\R^n \to \R^n$ is a quasiregular mapping, for $n\geq 2$, then $L(f)$ is a subset of the $(n-1)$-dimensional unit sphere $S^{n-1}$ in $\R^n$. 

\begin{definition}
Let $n\geq 2$ and $f\colon \R^n\rightarrow\R^n$ be a transcendental-type quasiregular mapping. We say $\zeta\in S^{n-1}$ is a Julia limiting direction of $f$ if there exists a sequence $(x_m)_{m=1}^{\infty} $ in $J(f)$ with $\lim\limits_{m\rightarrow\infty }|x_m|=\infty$ and $\lim\limits_{m\rightarrow\infty }x_m/|x_m|=\zeta$.
\end{definition}

As $J(f)$ is unbounded for transcendental-type quasiregular mappings, the notion of $L(f)$ as the radial distribution of $J(f)$ makes sense. 
In this setting, $L(f)$ was first studied by the first named author in \cite{Fletcher}. In \cite[Theorem 2.2]{Fletcher}, a strong condition on the comparability of the minimum modulus and maximum modulus of $f$ was given which implied that $L(f) = S^{n-1}$. The main result of this paper is a more direct analogue of Qiao's result relating the growth of a transcendental-type quasiregular mapping to the size of $L(f)$.

We denote by $\mu_f$ the order of growth of $f$ and by $m(E)$ the $(n-1)$-dimensional area of a subset $E$ of $S^{n-1}$. We also denote by $T(U)$ the topological hull of $U\subset \R^n$, the union of $U$ with all its complementary bounded components.

\begin{theorem}\label{Lf}
Let $f\colon \R^n\rightarrow\R^n$ be a transcendental-type $K$-quasiregular mapping of order $\mu_f<\infty$ for which $T(U)\neq \R^n$ for any quasi-Fatou component $U$. Then there exists a component $E\subset L(f)$ with
\begin{equation}
\label{eq:1}
m(E)\geq \min(c_n K^{(2-2n)/n}\mu_f^{1-n},\omega_{n-1})
\end{equation} 
where $c_n>0$ is a constant depending only on $n$ and $\omega_{n-1}=m(S^{n-1})$.
\end{theorem}

The requirement that $T(U)\neq \R^n$ for any quasi-Fatou component $U$ is always satisfied for transcendental entire functions in the plane. To see this, suppose there is a component $U\subset F(f)$ for a transcendental entire function $f$ such that $T(U)=\C$. Since transcendental entire functions have unbounded Julia sets, $U$ must be a multiply connected unbounded component of $F(f)$ which cannot happen by \cite[Theorem 1]{Baker}. It is still an open question whether the same is always true for transcendental-type quasiregular maps, but we need to make this assumption in order to apply a growth rate result in sectors contained in the quasi-Fatou set (see Theorem \ref{sectorLf} below).

We immediately obtain the following corollary from \eqref{eq:1}.

\begin{corollary} 
\label{cor:1}
Let $f\colon \R^n\rightarrow\R^n$ be a transcendental-type $K$-quasiregular mapping for which $T(U)\neq \R^n$ for any quasi-Fatou component $U$. With the notation from Theorem \ref{Lf}, if the order of $f$ satisfies
\[ \mu_f \leq K^{-2/n} \left ( \frac{ c_n}{\omega_{n-1} } \right )^{1/(n-1)} \] 
then $L(f)=S^{n-1}$.
\end{corollary}

This condition on the order of $f$ forces there to be a component $U$ of $L(f)$ with $m(U)\geq \omega_{n-1}$. Since $L(f)$ is closed, the closure of $U$ is also contained in $L(f)$ with $m(\overline{U})\geq \omega_{n-1}$. As there are no open sets of measure 0 in $S^{n-1}$, we have $L(f)=S^{n-1}$.

Using a similar proof strategy as in Theorem \ref{Lf}, we obtain the following corollary.

\begin{corollary}
    \label{cor:3} Let $f\colon \R^n\rightarrow\R^n$ be a transcendental-type $K$-quasiregular mapping for which $T(U)\neq \R^n$ for any quasi-Fatou component $U$. If every component of $L(f)$ is a point, then $\mu_f=\infty$.
\end{corollary}

The relevance of Corollary \ref{cor:3} is that if we can construct a quasiregular mapping with one Julia limiting direction that is the identity outside a half beam, we can completely solve the inverse problem: given a closed set $E\subset S^{n-1}$, we would like to find a transcendental-type quasiregular map of $f$ such that $L(f)=E$. As notation, let $H(x_0,\theta,w)\subset \R^n$ be a half-beam with width $2w>0$ along the center ray $R=\{x\in \R^n: \frac{ x-x_0}{|x-x_0|}=\theta\}$ starting at $x_0\in \R^n$ extending in the $\theta\in S^{n-1}$ direction.

\begin{theorem}\label{inverse}
Let $E\subset S^{n-1}$ be closed. Suppose that there exists a quasiregular mapping $f_0$ that satisfies 
\begin{enumerate}
 \item $L(f_0)=\{e_1\}$ 
 \item $f_0(x)=x$ for $x\in \R^n\backslash H(0,e_1,1)$.
\end{enumerate}
Then, there exists a quasiregular mapping $f$ with $L(f)=E$. 
\end{theorem}

\subsection{Phragm\'en-Lindel\"of principles}

The main novelty in our strategy for proving Theorem \ref{Lf} and Corollary \ref{cor:3} is a new version of the Phragm\'{e}n-Lindel\"{o}f principle which applies to quasiregular mappings in sectors and allows for boundary growth of the form $O(|x|^p)$ for $p>0$. To explain why this is important, we first recall that the classical Phragm\'{e}n-Lindel\"{o}f principle states that if $S$ is the sector $\{z\in \C:\left|\arg(z)\right|<\theta/2\}$ for $0<\theta\leq 2\pi$ and $f:S \to \C$ is a holomorphic function satisfying $\limsup\limits_{z\rightarrow w}|f(z)|\leq 1$ for all $w\in \partial S$ then either $|f(z)| \leq 1$ for all $z\in S$ or $\liminf\limits_{r\rightarrow\infty} \log M(r,f) r^{-\pi/\theta}>0$ where $M(r,f)$ is the maximum modulus function $\sup\limits_{\left|z\right|\leq r}\left|f(z)\right|$.

It is straight-forward to generalize this result to the case when $f$ satisfies $|f(z)| = O(|z|^p)$ on the boundary of the sector by considering the holomorphic function $f(z)/z^p$. This result for holomorphic functions is proved by using the corresponding Phragm\'en-Lindel\"of principle for subharmonic functions applied to the subharmonic function $\log|f|$.

Many Phragm\'{e}n-Lindel\"{o}f results have been proven for different classes of functions and more general unbounded domains. For example, analogous results have been given for sub-$F$-extremal maps \cite{GLM}, plurisubharmonic functions \cite{plurisub}, meromorphic functions \cite{mero}, slice regular functions \cite{slice}, and quasiregular mappings \cite{qrpl2,qrpl1,Rickman,qrpl3}. Other unbounded domains in $\R^n$ have also been considered such as strips \cite{d1}, cylinders \cite{d2}, and cones \cite{d3}. 

As far as the authors are aware, all of these generalizations of the Phragm\'en-Lindel\"of principle in the literature require a $O(1)$ bound on the boundary of the domain under consideration. For our purposes, we will need to consider the situation where a quasiregular mapping from a sector in $\R^n$, for $n\geq 2$, has a bound of the form $|f(x)| = O( |x|^p )$ on the boundary of the sector. Since there are no elementary algebraic considerations of the form $f(x) / x^p$ that can be used for quasiregular mappings, this generalization is not straight-forward. We expect our methods to be applicable in other settings too. 

Similar to how subharmonic functions are used to show Phragm\'{e}n-Lindel\"{o}f principles for holomorphic functions, we prove a Phragm\'{e}n-Lindel\"{o}f result for sub-$F$-extremal functions and apply to it get a corresponding result for quasiregular mappings. In the quasiregular setting, $\log |f|$ is an example of a sub-$F$-extremal function. We  postpone the recollection of this definition until the next section, but just note here that this means $\log |f|$ satisfies a certain non-linear partial differential equation and obeys a certain comparison principle, analogously to subharmonic functions. 

Next, we define what we mean by a sector in our context and recall a Phragm\'{e}n-Lindel\"{o}f result for sub-$F$-extremals that we generalize.

\begin{definition} 
\label{def:sector}
Let $E\subset S^{n-1}$ be a domain and let $x_0 \in \R^n$.
Define the sector relative to $E$ with vertex at $x_0$ to be $$S_{x_0,E}=\{c(x-x_0)+x_0: c>0,x\in E\}.$$ 
\end{definition}

The following result is \cite[Theorem 3.16]{GLM} in the special case of sectors. Note that the definitions used here will be made more precise in the next section.

\begin{theorem}(\cite{GLM}, Theorem 3.16)
Let $E\subsetneq S^{n-1}$ be a domain and $x_0\in \R^n$. Suppose that $S=S_{x_0,E}$ is a sector in $\R^n$ and that $F$ is a variational kernel of type $n$ in $S$ with structural constants $\alpha$ and $\beta$. Let $u\colon S\rightarrow \R\cup\{-\infty\}$ be a sub-$F$-extremal in $S$ such that $\limsup\limits_{x\rightarrow y} u(x)\leq 0$ for $y\in \partial S$. Then either $u(x)\leq 0$ in $S$ or
$$\liminf\limits_{r\rightarrow\infty}M_S(r,u)r^{-q} >0$$ where $M_S(r,u)=\sup\limits_{|x|\leq r,x\in S}u(x)$, $q=d_n m(E)^{-1/(n-1)}(\alpha/\beta)^{1/n}$, and $d_n>0$ is a constant depending only on $n$.
\end{theorem}

Applying this result to $\log |f|$ when $f$ is quasiregular, we have the following corollary.

\begin{corollary}
\label{cor:2}
Let $E\subsetneq S^{n-1}$ be a domain and $x_0\in \R^n$. Suppose that $S=S_{x_0,E}$ is a sector in $\R^n$. Let $f\colon S\rightarrow \R^n$ be a $K$-quasiregular mapping such that $\limsup\limits_{x\rightarrow y} |f(x)|\leq 1$ for $y\in \partial S$. Then either $|f(x)|\leq 1$ in $S$ or
$$\liminf\limits_{r\rightarrow\infty}\log M_S(r,f)r^{-q} >0$$ where $M_S(r,f)=\sup\limits_{|x|\leq r,x\in S}|f(x)|$, $q=d_n m(E)^{-1/(n-1)}K^{-2/n}$, and $d_n>0$ is a constant depending only on $n$.
\end{corollary}

Our generalizations of these results are as follows. Recall that $\log^+ x = \max \{ \log x, 0 \}$.

\begin{theorem}\label{PLFF}
Let $E\subsetneq S^{n-1}$ be a domain and $x_0\in \R^n$. Suppose that $S=S_{x_0,E}$ is a sector in $\R^n$ and that $F\colon S\times \R^n\rightarrow \R$ is a variational kernel of type $n$ in $S$ with structural constants $\alpha$ and $\beta$. Let $u\colon S\rightarrow \R\cup\{-\infty\}$ be a sub-$F$-extremal in $S$ such that $\limsup\limits_{x\rightarrow y} u(x)\leq C\log^+|y|$ at each $y\in \partial S$ for some constant $C>0$.  Then given $\varepsilon>0$, either for all sufficiently large $|x|$ in $S$ we have $$u(x) \leq C(1+\varepsilon) \log |x|$$ or 
$$\limsup\limits_{r\rightarrow\infty}M_S(r,u)r^{-q'}>0$$ where $q'=d_n m(E)^{-1/(n-1)}(\alpha/\beta)^{1/n}\varepsilon / (1+\varepsilon)$ and $d_n>0$ depends only on $n$.
\end{theorem}

This theorem weakens the hypotheses of \cite[Theorem 3.16]{GLM} by Granlund, Lindqvist and Martio from $u(x) = O(1)$ to $u(x)=O(\log |x|)$ on $\partial S$, although their conclusions in the alternative have a $\liminf$ instead of a $\limsup$. It is certainly conceivable that Theorem \ref{PLFF} could be improved, but for our dynamical applications this suffices.
Using the fact that if $f$ is a quasiregular mapping then $\log|f|$ is sub-$F$-extremal, we obtain the following corollary.

\begin{corollary}
\label{cor:pl} Let $E\subsetneq S^{n-1}$ be a domain and $x_0\in \R^n$.
Suppose that $S = S_{x_0,E}$ is a sector in $\R^n$. Let $f\colon S\rightarrow \R^n$ be a $K$-quasiregular mapping in $S$ such that 
\[ \limsup\limits_{x\rightarrow y} \log |f(x)|\leq p\log^+|y| + C\] 
at each $y\in \partial S$ for some constants $C$ and $p>0$.  Then given $\varepsilon >0$, either for all sufficiently large $|x|$ in $S$ we have
\[ \log |f(x)| \leq p(1+\varepsilon) \log |x| +C \]
or 
\[
\limsup\limits_{r\rightarrow\infty}\log M_S(r,f)r^{-q'}>0
\]
where $q'={d_n m(E)^{-1/(n-1)}K^{-2/n}}\varepsilon / (1+\varepsilon)$ and $d_n>0$ depends only on $n$.
\end{corollary}

The remainder of this paper is organized as follows. In section 2, we recall background material on quasiregular dynamics and nonlinear potential theory. In section 3, we prove Theorem \ref{PLFF}. In section 4, we prove a topological result needed in the proof of Theorem \ref{Lf}. Finally, in section 5 we prove our main results, Theorem \ref{Lf}, Corollary \ref{cor:3}, and Theorem \ref{inverse}.

We thank the referee for carefully reading the paper and providing many useful comments.

\section{Preliminaries}

As notation, we let $B(x_0,r)$ denote the open ball centered at $x_0 \in \R^n$ of radius $r$ and let $S(x_0,r)$ be the boundary of $B(x_0,r)$. The unit sphere in $\R^n$ is denoted by $S^{n-1}$.

\subsection{Quasiregular Dynamics}

We first recall some definitions about quasiregular mappings and their dynamics. For more information on quasiregular mappings, we refer the reader to Rickman's monograph \cite{Rickman}.

Let $G$ be a domain in $\R^n$ and $X=\R^n$ or $X=\R$. Denote by $C(G,X)$ the set of continuous functions $G\to X$. Denote by $W^{1,n}_\text{loc}(G,X)$ the Sobolev space of $\R^n$-valued or real-valued functions on $G$ that are locally in $L^n(G)$ with weak first order partial derivatives that are also locally in $L^n(G)$.

\begin{definition} Let $f\in C(G,\R^n)\cap W^{1,n}_\text{loc}(G,\R^n)$ for $G\subset \R^n$. We say $f$ is a quasiregular mapping if there exists a constant $K\geq 1$ such that $\left|f'(x)\right|^n\leq K J_f(x)$ for almost every $x\in \R^n$ where $\left|f'(x)\right|=\max\limits_{{\left|h\right|=1}}\left|f'(x)(h)\right|$ is the operator norm and $J_f(x)$ is the Jacobian of $f$ at $x$. The smallest such $K$ is called the outer dilatation of $f$ and is denoted $K_O(f).$
\end{definition}

If $f$ is quasiregular, there also exists a constant $K' \geq 1$ such that $J_f(x)\leq K' \min\limits_{\left|h\right|=1}\left|f'(x)(h)\right|^n$ almost everywhere. The smallest constant $K'$ here is called the inner dilatation of $f$ and is denoted by $K_I(f)$. The maximal dilatation is $K(f)=\max\{K_O(f),K_I(f)\}$. A $K$-quasiregular mapping is one where $K\geq K(f)$, so for example if $f$ is $2$-quasiregular, then it is also $3$-quasiregular.

We say that $f$ is an entire quasiregular mapping if $f$ is defined on $\R^n$ and that a non-constant entire quasiregular mapping $f$ is of transcendental-type if $f$ has an essential singularity at $\infty$. Otherwise we say that $f$ is of polynomial-type. 

Given an unbounded sub-domain $U$ of $\R^n$ and a function $f:U \to \R^n$, we define
\[ M_U(r,f) = \sup_{|x| \leq r, x\in U} |f(x) |,\]
as long as $\overline{B(0,r)}$ meets $U$. In the case where $U=\R^n$ and $f$ is quasiregular, the fact that $f$ is an open mapping means that we may simplify this notation and just write
\[ M(r,f) = \sup _{|x| = r} |f(x)|.\]
Transcendental-type quasiregular mappings may be classified via their growth as follows.

\begin{lemma}(\cite{Ber}, Lemma 3.4)
\label{lem:berg}
Let $f: \R^n \to \R^n$ be a quasiregular mapping. Then $f$ is of transcendental-type if and only if
\[ \lim\limits_{r\to \infty} \frac{ \log M(r,f) }{\log r} = + \infty .\]
\end{lemma}

A related notion to the expression above is the order of growth.

\begin{definition}
The order of an entire quasiregular mapping $f\colon \R^n\rightarrow\R^n$ is $$\mu_f=\limsup\limits_{r\rightarrow\infty}(n-1)\frac{\log\log M(r,f)}{\log r}.$$
\end{definition}

Next, consider the iteration of quasiregular mappings. If $f$ is an entire $K_1$-quasiregular mapping and $g$ is an entire $K_2$-quasiregular mapping, then the composition $f\circ g$ is also quasiregular, but the maximal dilatation typically goes up. In general, we have $K(f\circ g) \leq K_1K_2$. This means that the iterates $f^m$, for $m\in \N$, of a quasiregular mapping may be defined, but there is not necessarily a uniform bound on the maximal dilatation of the iterates. This is an obstacle to developing the theory of quasiregular dynamics in exact analogy as that of complex dynamics, as then the normal family machinery which is key in complex dynamics may not be available.

Nevertheless, the Julia set of a transcendental-type quasiregular mapping was defined by Bergweiler and Nicks \cite{BN} via a blowing-up property and has many of the properties expected of the Julia set.
As notation, let $\text{cap}(A,C)$ be the conformal capacity of a condenser as defined in \cite[p. 53]{Rickman} where $A$ is an open set in $\R^n$ and $C$ is a non-empty compact subset of $A$. It is known that if $\text{cap}(A,C)=0$ for some bounded open set $A$ containing $C$, then $\text{cap}(A',C)=0$ for any bounded open set $A'$ containing $C$. In this case, we use the notation $\text{cap}(C)=0$.

\begin{definition} 
Let $f\colon \R^n\rightarrow\R^n$ be a transcendental-type quasiregular mapping. The Julia set of $f$ is defined to be $$J(f)=\{x\in \R^n: \text{cap}\left(\R^n\backslash \cup_{k=1}^\infty f^k(U)\right)=0 \text{ for every neighborhood } U \text{ of } x\}.$$ The quasi-Fatou set is $QF(f)=\R^n\backslash J(f)$.
\end{definition}

We note that there is a version of this definition for polynomial-type mappings (see for example \cite{Ber2}), but there are extra complications that are not relevant to the current paper, so we say no more about this here.

For transcendental-type quasiregular mappings, $J(f)$ is unbounded (see for example \cite[Proposition 2.1]{Fletcher}), so we can study the radial distribution of $J(f)$, denoted by $L(f)$.
Again from \cite[Proposition 2.1]{Fletcher}, it follows that $L(f)$ is a closed, nonempty subset of $S^{n-1}$.

\subsection{Non-linear Potential Theory}

Next, we recall some definitions and properties from non-linear potential theory. For more details on this theory, we refer to Rickman's monograph \cite{Rickman} and Heinonen, Kilpel\"{a}inen, and Martio's book \cite{HKM}. As usual, suppose $G\subset \R^n$ is a domain. 

\begin{definition}
A function $F\colon G\times \R^n\rightarrow\R$ is called a variational kernel of type $n$ in $G$ if $F$ satisfies the following conditions.
\begin{itemize}
\item For each open $U\subset\subset G$ and every $\epsilon>0$ there exists a compact set $V\subset U$ with $m(U\backslash V)<\epsilon$ and $F\vert_{V\times \R^n}$ is continuous. 
\item For almost every $x\in G$ the function $h\mapsto F(x,h)$ is strictly convex and continuously differentiable.
\item There exist positive constants $\alpha$ and $\beta$ such that for almost every $x\in G$, 
\begin{equation}
\label{eq:2}
\alpha|h|^n\leq F(x,h)\leq \beta|h|^n
\end{equation}
for all $h\in \R^n$.
\item For almost every $x\in G$, we have $$F(x,\lambda h)=|\lambda|^n F(x,h),$$ $\lambda\in \R, h\in \R^n$.
\end{itemize}
\end{definition}

Denote by $\alpha(F)$ the largest such constant $\alpha$ and by $\beta(F)$ the smallest such constant $\beta$ such that \eqref{eq:2} holds. These are called the structural constants for $F$. We note that variational kernels of type $p\neq n$ have been studied, but we will focus only on the $p=n$ case. Henceforth, all variational kernels will be of type $n$.

Next, we recall the definitions of $F$-extremals, sub-$F$-extremals, and super-$F$-extremals in $G$ and list some properties concerning them. These are the generalizations of harmonic, subharmonic, and superharmonic functions to this setting. 

\begin{definition}
A real-valued function $u\in W_{n,\text{loc}}^1(G,\R)$ is called an $F$-extremal if $u$ is a solution of the Euler equation $\nabla\cdot \nabla_h F(x,\nabla u)=0$ in the weak sense, i.e. $$\int_G \nabla_h F(\cdot,\nabla u)\cdot \nabla \phi \: dm=0$$ for all real-valued functions $\phi\in C_0^\infty(G,\R)$.
\end{definition}

\begin{definition} 
An upper semi-continuous function $u\colon G\rightarrow\R\cup\{-\infty\}$ is called a sub-$F$-extremal if $u$ satisfies the following condition: if a domain $U$ is relatively compact in $G$ and $h\in C(\overline{U},\R)$ is an $F$-extremal in $U$ with $h\geq u$ in $\partial U$ then $h\geq u$ in $U$. A function $v\colon G\rightarrow\R\cup\{\infty\}$ is called a super-$F$-extremal if $-v$ is a sub-$F$-extremal. 
\end{definition}

Given an unbounded sub-domain $U$ of $\R^n$ such that $\overline{B(0,r)}\cap U\neq \emptyset$ and a sub-$F$-extremal $u\colon U\rightarrow\R$, we define $$ M_U(r,u)=\sup_{|x|\leq r,x\in U}u(x).
$$

\begin{lemma}\label{propSubSuper} Let $h$ be an $F$-extremal, $u$ be a sub-$F$-extremal, $v$ be super-$F$-extremal, and $\lambda\in \R$. Then, we have the following properties:

\begin{itemize}
\item $h$ is a sub-$F$-extremal and a super-$F$-extremal,

\item $\lambda h$ and $h+\lambda$ are $F$-extremals,

\item $|\lambda| u$ and $u+\lambda$ are sub-$F$-extremals, and

\item $|\lambda| v$ and $v+\lambda$ are super-$F$-extremals.
\end{itemize}
\end{lemma}

The first property listed in Lemma \ref{propSubSuper} is a consequence of \cite[Theorem 4.18]{GLM2}. 
The following $F$-Comparison Principle is key to proving Phragm\'en-Lindel\"of-type results.

\begin{theorem}(\cite{GLM3}, Lemma 2.3)
\label{CompPrinc} 
Let $G$ be a bounded domain, $u$ be a sub-$F$-extremal in $G$, and $v$ be a super-$F$-extremal in $G$. If $$\limsup\limits_{x\rightarrow y}u(x)\leq \liminf\limits_{x\rightarrow y}v(x)$$ for all $y\in \partial G$ and if the left and right hand sides of the inequality are neither $\infty$ nor $-\infty$ at the same time then $u\leq v$ in $G$.
\end{theorem}

The connection between variational kernels and quasiregular mappings is revealed by the following definition.

\begin{definition} 
\label{def:sharp}
Let $f\colon G\rightarrow\R^n$ be a quasiregular mapping. Suppose $G'\subset \R^n$ is a domain such that $f(G)\subset G'$. Let $F\colon G'\times \R^n\rightarrow \R$ be a variational kernel in $G'$. Define $f^\sharp F\colon G\times \R^n\rightarrow \R$ as 
$$f^\sharp F(x,h)=
\begin{cases}
        F(f(x),J_f(x)^{1/n}f'(x)^{-1*}h), &  J_f(x)\neq 0\\
        |h|^n, & J_f(x)=0 \text{ or } J_f(x) \text{ does not exist.}
        \end{cases}$$ where $A^*$ is the adjoint of a linear map $A\colon\R^n\rightarrow \R^n.$
\end{definition}

By \cite[Proposition VI.2.6]{Rickman}, we know that $f^\sharp F$ is a variational kernel in $G$ with possibly different constants $\alpha$ and $\beta$. In fact, we can take $\alpha=\alpha(F)K_O(f)^{-1}$ and $\beta=\beta(F)K_I(f)$. However, we may have $\alpha(f^\sharp F)\geq \alpha(F)K_O(f)^{-1}$ and $\beta(f^\sharp F)\leq \beta(F)K_I(f)$. 
In the special case when $F_I(x,h):=|h|^n$, $\alpha(f^\sharp F_I)=K_O(f)^{-1}$ and $\beta(f^\sharp F_I)=K_I(f)$. In particular, if $f:G\to \R^n$ is quasiregular then \cite[Corollary VI.2.8]{Rickman} states that $\log |f|$ is $f^{\sharp}F_I$ extremal in $G \setminus f^{-1}(0)$ which implies that $f$ is sub-$f^\sharp F_I$-extremal in $G$.

Next, we recall the definition of $F$-harmonic measure $\omega(C,G;F)$ of a set $C\subset \partial G$.

\begin{definition} Let $G\subset \R^n$ be a bounded domain, $F$ a variational kernel, and $f\colon \partial G\rightarrow \R\cup \{\infty,-\infty\}$ be any function. The upper Perron class associated to $f$ is defined as $\mathscr{U}_f=\{u\colon G\rightarrow\R\cup \{\infty\}:\liminf\limits_{x\rightarrow y}u(x)\geq f(y), y\in \partial G, u \text{ is a super-$F$-extremal and bounded from below in } G\}.$ Define the function $\overline{H_f}$ by $\overline{H_f}=\inf\{u:u\in \mathscr{U}_f\}$. 
\end{definition}

Note that if $f$ is bounded, then $\overline{H_f}$ is an $F$-extremal in $G$, see \cite[p.106]{GLM}.

\begin{definition}
Let $G\subset \R^n$ be a bounded domain and $C\subset \partial G$. Let $f\colon \partial G\rightarrow \{0,1\}$ be the characteristic function of $C$ in $\partial G$. The $F$-extremal $\overline{H_f}$ in $G$ is called the $F$-harmonic measure of $C$ with respect to $G$ and denoted $\omega(C,G;F)$.
\end{definition}

From the definition, we see that $0\leq \omega(C,G;F)\leq 1$.

\subsection{Sectors}

For our applications, we will need estimates on $\omega(C,G ; F)$ when $G$ is a subset of a sector in $\R^n$ and on the growth of quasiregular mappings in sectors in the quasi-Fatou set. Recall the definition of $S_{x_0,E}$ from Definition \ref{def:sector}.
It will also be useful to project points from sectors to $S^{n-1}$.

\begin{definition}
Given $x_0\in \R^n$, define the projection $\nu_{x_0} : \R^n \setminus \{x_0 \} \to S^{n-1}$ via
\[ \nu_{x_0}(x) = \frac{ x-x_0}{|x-x_0|} .\]
\end{definition}

If the context is clear, we typically suppress the subscripts and write $S = S_{x_0,E}$ and $\nu = \nu_{x_0}$.

The important feature of sectors is that they have constant angle measure. More precisely, if $S_{x_0,E}$ is a sector with vertex $x_0$ then the $(n-1)$-dimensional measure of $\nu (S_{x_0,E}\cap S(x_0,r) ) = m(E)$ is constant for all $r>0$. 

First, we have the following estimate for $F$-harmonic measure in a sector. To start, we fix a sector $S=S_{x_0,E}$ with vertex $x_0$. For $r>0$, let $S_r = S \cap B(x_0,r)$.
Then denote the value of the $F$-harmonic measure $\omega(S(x_0,r)\cap \partial S_r, S_r;F)$ at $x\in S_r$ by $\omega(x;r)$.
Using the above notation, we have the following estimate on $\omega(x;r)$.

\begin{lemma}\label{318} 
Let $S_{x_0,E}\subset \R^n$ be a sector and $F$ be a variational kernel in $S_{x_0,E}$ with structural constants $\alpha$ and $\beta$. Then, 
\[ \omega(x;r)\leq \frac{4|x-x_0|^{q}}{r^{q}}  \] 
where $q=d_n m(E)^{-1/(n-1)}(\alpha/\beta)^{1/n}$ and $d_n>0$ is a constant depending only on $n$. 
\end{lemma}

\begin{proof}
In the case where the vertex $x_0=0$, the result is just a special case of \cite[Lemma 3.18]{GLM}. Let $f$ be the translation $f(x) = x+x_0$ and consider the variational kernel $f^{\sharp}F$. By Definition \ref{def:sharp}, it is clear that 
\[ f^{\sharp}F(x,h) = F(f(x),h)\]
and that $\alpha(f^{\sharp}F) = \alpha (F), \beta(f^{\sharp}F) = \beta(F)$. It follows that
\[ \omega(S(x_0,r)\cap \partial S_r, S_r;F)  = \omega ( f^{-1}(S) \cap S(0,r) , f^{-1}(S) \cap B(0,r) ; f^{\sharp}F) \]
from which we obtain the lemma, again by \cite[Lemma 3.18]{GLM}.
\end{proof}

Next, we will need the following growth condition on $f$ if a sector is contained in the quasi-Fatou set. To state it, it will help to have the following notation.
Denote by $d$ the spherical distance on $S^{n-1}$.
For $x_0 \in \R^n$, $\zeta \in S^{n-1}$ and $0<\eta \leq \pi $, let $\Omega ( x_0 , \zeta, \eta )$ be the sector $S_{x_0,E}$ where $E = \{ y \in S^{n-1} : d(y,\zeta) < \eta \}$.

\begin{theorem}(\cite{Fletcher}, Theorem 4.2) \label{sectorLf}
Let $f\colon \R^n\rightarrow \R^n$ be a transcendental-type $K$-quasiregular mapping. Suppose that for $x_0\in \R^n$, $\theta\in S^{n-1}$, and $\eta>0$ the sector $\Omega(x_0,\theta,\eta)$ is contained in a component $U$ of the quasi-Fatou set of $f$ for which the topological hull of $U$ is a proper subset of $\R^n$. Then, if $0<\eta'<\eta$, there exists a constant $p$ depending on $n, \eta-\eta'$ and $K$ such that $\vert f(x)\vert=O(|x|^p)$ for $x\in \Omega(x_0,\theta,\eta')$.
\end{theorem}

\section{Phragm\'{e}n-Lindel\"{o}f Principles}

In this section, we prove our Phragm\'{e}n-Lindel\"{o}f result for sub-$F$-extremals in a sector where $u$ is bounded above by $C\log^+ |x|$ for some constant $C>0$ and then apply it to get a corresponding result for quasiregular mappings.

\begin{proof}[Proof of Theorem \ref{PLFF}]

Recall that we are assuming that $\limsup\limits_{x\to y} u(x) \leq  C \log^+ |y|$ for all $y\in \partial S$. Suppose that the first alternative in the conclusions of Theorem \ref{PLFF} does not hold. That is, suppose that given $\varepsilon >0$ there exists a sequence $(x_m)_{m=1}^{\infty}$ in $S$ with $|x_m| \to \infty$ and $u(x_m) > (1+\varepsilon) C \log|x_m|$. As $|x_m| \to \infty$, we may re-label and assume that 
\begin{equation}
\label{eq:plff0} 
\frac{u(x_m)}{\log|x_m|}> C(1+\varepsilon) 
\end{equation}
for all $m$.
Given $\delta >0$, define 
\[ C_m=\frac{u(x_m)}{(1+\varepsilon)\log|x_m|} \text{ and } R_m=|x_m|^{1+\varepsilon}e^{-\delta/C_m}-|x_0|. \] 
Clearly we have $C_m > C$ for all $m$. Also, re-labeling $m$ if needed, we have
\begin{align*}
\log(R_m+|x_0|)&=(1+\varepsilon)\log|x_m|-\frac{\delta}{C_m}\\
&=(1+\varepsilon)\log|x_m|-\frac{(1+\varepsilon)\log|x_m|\delta}{u(x_m)}\\
&=(1+\varepsilon)\log|x_m|\left(1-\frac{\delta}{u(x_m)}\right)>0
\end{align*}
for all $m$ since $|x_m|\rightarrow\infty$ and \eqref{eq:plff0} imply that $\log|x_m|>0$ and $u(x_m)\rightarrow\infty$.

Keeping the notation from earlier, let $S_{R_m}=S\cap B(x_0,R_m)$. Finally, define
\[ v_m(x)=u(x)-C_m\log(\left|x_0\right|+R_m). \]
The aim is to show that a suitably scaled version of $v_m$ is bounded above by $F$-harmonic measure on $\partial S_{R_m}$.
Towards that end, for $y \in \partial S_{R_m}\cap \partial S$, we have 
\begin{equation}
\label{eq:plffeq1}
\limsup\limits_{x\to y} v_m(x)= \limsup\limits_{x\to y} \left [ u(x)-C_m\log(\left|x_0\right|+R_m) \right ] \leq \limsup\limits_{x\to y } \left [ u(x)-C\log^+\left|x\right| \right ]\leq 0.
\end{equation}

Moreover, for $y\in \partial S_{R_m}\backslash \partial S$, we have 
\begin{equation}
\label{eq:plffeq2}
v_m(y)=u(y)-C_m\log(\left|x_0\right|+R_m)\leq u(y),
\end{equation}
noting that $u$, and hence $v_m$, are defined on $\partial S_{R_m}\backslash \partial S$.

Consider $\frac{v_m(x)}{M_S(\left|x_0\right|+R_m,u)}$ where as usual $M_S(r,u)=\sup\limits_{x\in S, |x|\leq r}u(x)$. 
It follows from \eqref{eq:plff0} that $u(x_m)>0$, so we have $M_S(|x_0| + R_m,u)>0$. 
For $y\in \partial S_{R_m}\cap \partial S$, it follows from \eqref{eq:plffeq1} that 
\[ \limsup\limits_{x\to y } \frac{v_m(x)}{M_S(\left|x_0\right|+R_m,u)}\leq 0.\]
For $y\in \partial S_{R_m}\backslash \partial S$, it follows from \eqref{eq:plffeq2} that
\[ \frac{v_m(y)}{M_S(\left|x_0\right|+R_m,u)}\leq \frac{u(y)}{M_S(\left|x_0\right|+R_m,u)}\leq 1.\] 

From \cite[Theorem 11.6]{HKM}, we have $$\lim_{x\rightarrow y}\omega(x;R_m)=\begin{cases}
    0, &y\in \partial S_{R_m}\backslash S(x_0,r)\\
    1, &y\in \partial S_{R_m}\backslash \partial S 
\end{cases}.$$ Since $0\leq \omega(x,R_m)\leq 1$ for all $x\in \partial S_{R_m}$, we have $$\limsup\limits_{x\rightarrow y}\frac{v_m(x)}{M_S(\left|x_0\right| + R_m,u)}\leq\liminf\limits_{x\rightarrow y} \omega(x;R_m)$$ for $y\in\partial S_{R_m}.$

As $\omega(x;R_m)$ is an $F$-extremal in $S_{R_m}$, we also have that $\omega(x;R_m)$ is a super-$F$-extremal in $S_{R_m}$ by Lemma \ref{propSubSuper}. 
As $u(x)$ is a sub-$F$-extremal, again by Lemma \ref{propSubSuper}, we have that both $v_m(x)$ and $
\frac{v_m(x)}{M_S(|x_0|+R_m,u)}$ are sub-$F$-extremals in $S_{R_m}$. 

By the $F$-Comparison Principle, Theorem \ref{CompPrinc}, we have 
\[ \frac{v_m(x)}{M_S(\left|x_0\right| + R_m,u)}\leq \omega(x;R_m)\] 
for $x\in S_{R_m}$. In particular, we have
\[ v_m(x_m)\leq M_S(\left|x_0\right| + R_m,u)\omega(x_m;R_m).\]

Next, by the definitions of $C_m$ and $R_m$, for all $m$ we have 
\begin{align*}
v_m(x_m)&=u(x_m)-C_m\log(\left|x_0\right|+R_m)\\&=u(x_m)-C_m\log(\left|x_0\right|+\left|x_m\right|^{1+\varepsilon}e^{-\delta/C_m}-|x_0|)\\
&=u(x_m)-C_m\left [ \log(\left|x_m\right|^{1+\varepsilon})+\log(e^{-\delta/C_m}) \right ]\\
&=u(x_m)-(1+\varepsilon)C_m\log\left|x_m\right|+\delta\\
&=u(x_m)-\frac{(1+\varepsilon)u(x_m)}{(1+\varepsilon)\log|x_m|}\log\left|x_m\right|+\delta\\
&=\delta>0.
\end{align*}

We conclude that for all $m$, we have
\[ 0<\delta\leq M_S(\left|x_0\right| + R_m,u)\omega(x_m;R_m).\]
Using the estimate for $F$-harmonic measure in sectors from Lemma \ref{318}, for all $m$ we get
\begin{equation}
\label{eq:plff3} 
\delta \leq \frac{ 4M_S(|x_0| + R_m,u) |x_m - x_0| ^q }{R_m^q} ,
\end{equation}
where $q=d_n m(E)^{-1/(n-1)}(\alpha/\beta)^{1/n}$ for some constant $d_n>0$ which depends only on $n$. Here, $\alpha$ and $\beta$ are the structural constants of $F$, and $m(E)$ is the angle measure of $S$. 

Set $r_m = |x_0| + R_m$. 
Thus $R_m = r_m(1+o(1))$ as $m\to \infty$. Moreover, from the definition of $R_m$, we have
\begin{align*} |x_m| &= \left [ (R_m + |x_0|)e^{\delta / C_m} \right ] ^{1/(1+\varepsilon)} \\
&= O( r_m^{1/(1+\varepsilon) } )
\end{align*}
as $m\to \infty$. Therefore 
\[ |x_m-x_0|^q = O( r_m^{q/(1+\varepsilon)} ) \]
as $m\to \infty$. 

Combining these with \eqref{eq:plff3}, we conclude that there exists a constant $\lambda >0$ such that
\[ \frac{ M_S(r_m,u)r_m^{q/(1+\varepsilon)} }{r_m^{q} } = M_S(r_m,u)r_m^{-q\varepsilon/(1+\varepsilon) } \geq \lambda \]
for all $m$.
Therefore,
\[ \limsup\limits_{r\to \infty} M_S(r,u) r^{-q'} > 0 \]
with $q' = d_n m(E)^{-1/(n-1)}(\alpha/\beta)^{1/n} \varepsilon / (1+\varepsilon)$.

\end{proof}

We now apply Theorem \ref{PLFF} to the special case of $\log |f|$ when $f$ is quasiregular in $S$.

\begin{proof}[Proof of Corollary \ref{cor:pl}]
Let $F_I(x,h)=|h|^n$ be the variational kernel. Then $u(x)=\log|f(x)|$ is a sub-$f^\sharp F_I$-extremal in $S$ with $\alpha(f^\sharp F_I)=K_O(f)^{-1}$ and $\beta(f^\sharp F_I)=K_I(f)$. By the hypotheses, 
at each $y\in \partial S$, we have
$$\limsup\limits_{x\rightarrow y}u(x)=\limsup\limits_{x\rightarrow y}\log|f(x)|\leq p\log^+|y|+C.$$ 
Therefore,
$$\limsup\limits_{x\rightarrow y}\left [ u(x)- C \right ] \leq p\log^+|y|.$$

As $u(x) - C$ is a sub-$f^{\sharp}F_I$-extremal by Lemma \ref{propSubSuper}, then by Theorem \ref{PLFF} given $\varepsilon>0$ we obtain either
$u(x)- C\leq (1+\varepsilon) p \log|x|$ in $S$ for $|x|$ sufficiently large or there exists a constant $q=d_n m(E)^{-1/(n-1)}(\alpha/\beta)^{1/n} \varepsilon / (1+\varepsilon)>0$ such that $\limsup\limits_{r\rightarrow\infty}M_S(r,u)r^{-q}>0$. Note that $\alpha/\beta>K^{-2}$, so $$\limsup\limits_{r\rightarrow\infty}M_S(r,u)r^{-q'}>\limsup\limits_{r\rightarrow\infty}M_S(r,u)r^{-q}>0$$ where $q'={d_n m(E)^{-1/(n-1)}K^{-2/n}}\varepsilon / (1+\varepsilon).$

In the first case, $u(x)\leq C+(1+\varepsilon)p\log |x|$ in $S$ for $|x|$ sufficiently large.
In the second case, if $\limsup\limits_{r\rightarrow\infty} M_S(r,u)r^{-q'} >0$ then as
\[ M_S(r,u) = M_S(r , \log |f|) = \log M_S(r,f)\]
we conclude that
\[ \limsup\limits_{r\to\infty} [\log M_S(r,f)] r^{-q'} >0 .\]
\end{proof}

\section{Neighborhoods of Components of Closed Sets}
In the proof of Theorem \ref{Lf}, we will need to show that every component $E$ of $L(f)$ has an open neighborhood $U$ which is not too much larger than $E$ and has the property that the boundary of $U$ is contained in the complement of $L(f)$. We will do this in Lemma \ref{toplemma}, and to build up to its proof, we will introduce some topological notions here that are not required anywhere else in this paper. For more details and proofs, we refer to Bredon's monograph \cite{Bredon}.

\begin{definition}
    Let $X$ be a topological space. Define $x\sim y$ if $X$ cannot be written as a disjoint union of open sets $U$,$V$ containing $x,y$ respectively. Then $\sim$ is an equivalence relation, and equivalence classes are quasicomponents of $X$.
\end{definition}

\begin{theorem}\label{F1}
    The quasicomponent containing $x\in X$ is the intersection of closed-and-open subsets of $X$ containing $x$.
\end{theorem}

\begin{theorem}\label{F2}
    In compact Hausdorff spaces, components and quasicomponents coincide. 
\end{theorem}

\begin{definition}
    A space $X$ is $T_4$ if for any two disjoint closed sets $F,G$ in $X$, there exist disjoint open sets $U,V$ in $X$ with $F\subset U$ and $G\subset V$.
\end{definition}

\begin{theorem}\label{F3}
    A set $A$ is measurable if and only if for every $\epsilon>0$ there exist an open set $G$ and a closed set $H$ such that $H\subset A\subset G$ and $m(G\backslash H)<\epsilon$.
\end{theorem}

\begin{lemma}
    \label{other lemma} Let $X$ be a compact Hausdorff space. Let $E$ be a quasicomponent of $X$, and let $V$ be an open neighborhood of $E$ in $X$. Then there exists a closed-and-open set $C$ in $X$ such that $E\subset C\subset V$.
\end{lemma}
\begin{proof}
Since $V$ is open in $X$, $X\backslash V$ is closed in $X$. Since $X\backslash V$ is closed in a closed Hausdorff space, then $X\backslash V$ is compact. 

    Let $x\in E$ and let $Y_\gamma$ be closed-and-open sets in $X$ containing $x$. Then, $\cap_\gamma Y_\gamma=E$ by Theorem \ref{F1}. Since $Y_\gamma$ is closed in $X$, then $X\backslash Y_\gamma$ is open in $X$. Therefore, $\cup_\gamma (X\backslash Y_\gamma)$ is open. Let $y\in X\backslash V$. Then $y\not\in E$, so $y\not\in Y_\gamma$ for some $\gamma$. Hence $y\in \cup_\gamma(X\backslash Y_\gamma)$, so $X\backslash V\subset \cup_\gamma(X\backslash Y_\gamma)$. As $\{X\backslash Y_\gamma\}_\gamma$ is an open cover of $X\backslash V$, there is a finite subcover $\{X\backslash Y_i\}_{i=1}^m$ such that $X\backslash V\subset \cup_{i=1}^m (X\backslash Y_i)$. 

    Since $Y_i$ is closed-and-open for all $i$, $\cup_{i=1}^m(X\backslash Y_i)$ is closed-and-open in $X$, so $$X\backslash (\cup_{i=1}^m(X\backslash Y_i))$$ is closed-and-open in $X$. Let $x\in E$. Then, $x\in Y_\gamma$ for all $\gamma$. In particular, $x\not\in X\backslash Y_i$ for all $i$, so $x\not\in \cup_{i=1}^m (X\backslash Y_i)$. Hence $E\subset X\backslash (\cup_{i=1}^m (X\backslash Y_i))$. Finally, since $\cup_{i=1}^m (X\backslash Y_i)$ is an open cover of $X\backslash V$, then $X\backslash (\cup_{i=1}^m (X\backslash Y_i))\subset V$.
\end{proof}

\begin{lemma}\label{toplemma}
    Let $X\subset S^{n-1}$ be closed, and let $E$ be a component of $X$ with $m(E)<\lambda<\omega_{n-1}$. Given $\epsilon>0$, there exists a neighborhood $U$ of $E$ with $m(U)<\lambda+\epsilon<\omega_{n-1}$ and $\partial U\subset S^{n-1}\backslash X$.
\end{lemma}

\begin{proof}
    Since $E$ is closed in $S^{n-1}$, $E$ is a measurable set. By Theorem \ref{F3}, there exist an open set $V_1\subset S^{n-1}$ and a closed set $V_2\subset S^{n-1}$ such that $V_2\subset E\subset V_1$ and $m(V_1\backslash V_2)<\epsilon$. Note that this implies that $m(V_1)=m(V_2\cup (V_1\backslash V_2))=m(V_2)+m(V_1\backslash V_2)\leq m(E)+m(V_1\backslash V_2)<\lambda+\epsilon$. If $X\cap \partial V_1=\emptyset$, we are done.

    Suppose $X\cap \partial V_1\neq \emptyset$. Note that $V_1$ is open in $S^{n-1}$, so $X\cap V_1$ is open in $X$ and that $E\subset X\cap V_1$. By Lemma \ref{other lemma}, there exists a closed-and-open set $C_1$ in $X$ with $E\subset C_1\subset X\cap V_1$. Since $C_1$ is closed in $X$ and $X$ is closed in $S^{n-1}$, then $C_1$ is closed in $S^{n-1}$. Since $C_1$ is also open in $X$, $X\backslash C_1$ is closed in $X$ and hence in $S^{n-1}$. In addition, since $V_1$ is open in $S^{n-1}$, then $S^{n-1}\backslash V_1$ is closed in $S^{n-1}$. Therefore, $C_2=(S^{n-1}\backslash V_1)\cup (X\backslash C_1)$ is closed in $S^{n-1}$. 
    
    Note that $C_1\cap C_2=\emptyset$. To see this, suppose that $x\in C_2$. Then, either $x\in S^{n-1}\backslash V_1$ or $x\in X\backslash C_1$. In the first case, since $C_1\subset V_1$, $x$ cannot be in $C_1$. In the second case, $x$ cannot be in $C_1$. Hence, $C_1\cap C_2=\emptyset$. Since $S^{n-1}$ is $T_4$, there exist disjoint open sets $U_1, U_2$ in $S^{n-1}$ with $C_1\subset U_1$ and $C_2\subset U_2$. 

    Note that $U_1\subset V_1$. To see this, let $x\in U_1$. Then, $x\not\in U_2$, so $x\not\in C_2$. Therefore, $x\not\in X\backslash C_1$ and $x\not\in S^{n-1}\backslash V_1$. Hence, $x\in V_1$. Therefore, $m(U_1)\leq m(V_1)<\lambda+\epsilon$.

    Finally, note that $\partial U_1\cap X=\emptyset$. To see this, let $x\in X=C_1\cup (X\backslash C_1)$. If $x\in C_1$, then $x\subset U_1$ and $x\not\in \partial U_1$. If $x\in X\backslash C_1$, then $x\in U_2$. Since $U_2$ is open, there exists a neighborhood of $x$ contained in $U_2$. If $x\in \partial U_1$, any neighborhood of $x$ contains a point in $U_1$. However $U_1\cap U_2=\emptyset$, so this cannot happen. Therefore, $\partial U_1\cap X=\emptyset$.
    \end{proof}

\section{Proof of Main Results}

In this section, we return to the set of Julia limiting directions for a quasiregular mapping and prove Theorem \ref{Lf}, Corollary \ref{cor:3}, and Theorem \ref{inverse}.
Our aim is to generalize the method of Qiao \cite[Theorem 1]{Qiao} using our Phragm\'en-Lindel\"of result to overcome technical difficulties in higher dimensions. We first need the following lemma. Recall that $\Omega(x_0,\zeta, \eta)$ is the sector $S_{x_0,E}$ where $E$ is the ball $\{ y \in S^{n-1} : d(y,\zeta) < \eta) \}$.

\begin{lemma}
\label{lem:qf}
Suppose $f:\R^n \to \R^n$ is a transcendental-type quasiregular mapping.
Let $\zeta \in S^{n-1} \setminus L(f)$. Then there exist $k>0$ and $\eta >0$ such that the sector $\Omega (k \zeta, \zeta, \eta) \subset QF(f)$.
\end{lemma}

\begin{proof}
Suppose the conclusion does not hold. Then for all $k>0$ and all $\eta >0$, we have
\[ \Omega ( k\zeta, \zeta, \eta) \cap J(f) \neq \emptyset.\]
Note that $\nu_{0}(\Omega(k\zeta ,\zeta,\eta))\subset S^{n-1}\cap B(\zeta, \eta)$. 
Set $\eta_m = 1/m$ and choose sequences $k_m \to \infty$ and $z_m \in \Omega ( k_m \zeta, \zeta, \eta_m) \cap J(f)$. We must necessarily have $|z_m| \to \infty$ as $m\to \infty$.

For all $\epsilon >0$, there exists $M \in \N$ such that $\eta_m < \epsilon$ for $m\geq M$. Then 
\[ \nu_0(z_m) \in \nu_0 \left [ \Omega ( k_m \zeta, \zeta, \eta_m ) \right ]\subset B(\zeta , \epsilon ).\]
As this holds for all $\epsilon >0$, we must have $\nu_0(z_m) \to \zeta$ as $m\to \infty$. This implies that $\zeta \in L(f)$, which is a contradiction.
\end{proof}

\begin{proof}[Proof of Theorem \ref{Lf}]
Let $f$ be a $K$-quasiregular mapping of transcendental-type of order $\mu_f$, and let $$M=\min(c_nK^{(2-2n)/n}\mu_f^{1-n},\omega_{n-1})=\min\left(\left[(n-1)d_nK^{-2/n}\mu_f^{-1}/2\right]^{n-1},\omega_{n-1}\right)$$ where $d_n$ is the constant depending on $n$ from Corollary \ref{cor:pl} and $c_n=[(n-1)d_n/2]^{n-1}$. Assume towards a contradiction that the component of $L(f)$ with the largest measure has measure less than $M$.

Let $\mathcal{F} =\{U : U \text{ is a domain in }S^{n-1} \text{ such that } \partial U\subset S^{n-1}\backslash L(f) \text{ and }m(U)<M\}.$
Since the component with the largest measure of $L(f)$ has measure less than $M$, then by Lemma \ref{toplemma}, $\mathcal{F}$ is an open cover of $S^{n-1}$. Since $S^{n-1}$ is compact, there exists a finite subcover $\bigcup_{i=1}^s U_i$ of $S^{n-1}$, for $U_1,\ldots, U_s \in \mathcal{F}$. 

Now, fix $i\in \{1,2,...,s\}$ and consider $U_i$. Let $S_i$ be the sector $S_{0,U_i}$.
Given $y\in \partial U_i$, by Lemma \ref{lem:qf} there exist $k_y>0$ and $\eta_y >0$ such that $\Omega(k_yy, y, \eta_y) \subset QF(f)$.
By Theorem \ref{sectorLf}, there exist constants $\eta_y'\in (0,\eta_y)$ and $p_y >0$ such that $|f(x)|=O(|x|^{p_{y}})$ for $x\in \Omega(k_yy,y,\eta_{y}')$.

Define $V_y=\nu_0[\Omega(k_yy,y,\eta_y')]\subset S^{n-1}$, and consider $\bigcup_{y\in \partial U_i} V_y$. This is an open cover of $\partial U_i$. By compactness, there exists a finite subcover $\bigcup_{j=1}^t V_{y_j}$ that covers $\partial U_i$. Set $p_i=\max\limits_{j=1,\ldots, t}p_{y_j}$.

Next, for $x\in \partial S_i$, with $|x|>R$ for $R$ sufficiently large, we have $x\in \Omega(k_{y_j}y_j,y_j,\eta_{y_j}')$ for some $y_j$, and $\nu_0(x)\in V_{y_j}$. Therefore,
\[ x \in \bigcup_{j=1}^t S_{k_{y_j}y_j,V_{y_j}}\] where $S_{k_{y_j}y_j,V_{y_j}}$ is the sector with vertex $k_{y_j}y_j$ with respect to $V_{y_j}$. 
Hence, as $|x|\to \infty$ in $\partial S_i$, we have $|f(x)| = O(|x|^{p_i})$. Moreover, as $\overline { \partial S_i \cap B(0,R) }$ is compact and $f$ is continuous on this set, $|f|$ is bounded there. We may therefore apply Corollary \ref{cor:pl} to $f$ on $S_i$ with $\varepsilon =1$ to conclude that there either there exists $p_i' >0$ so that $|f(x)| = O(|x|^{p_i'} )$ as $|x| \to \infty$ in $S_i$, or 
\[ \limsup\limits_{r\to\infty} \log M_{S_i}(r,f) r ^{-q_i} >0,\] 
where $q_i=d_n m(U_i)^{-1/(n-1)}K^{-2/n}/2$ for $d_n>0$ depending only on $n$.

Suppose this latter case holds. Then, in particular, there exist a constant $\lambda >0$ and a sequence $(r_m)_{m=1}^{\infty}$ with $r_m \to \infty$ such that
\[ \log M_{S_i}(r_m,f)>\lambda r_m^{q_i}.\]

Then the order of $f$ satisfies
\begin{align}
\mu_f=&\limsup\limits_{r\rightarrow\infty}\: (n-1)\frac{\log\log M(r,f)}{\log r}\nonumber\\
&\geq \limsup\limits_{m\to \infty} \: (n-1) \frac{ \log \log M(r_m,f)}{\log r_m}\nonumber\\
&\geq \limsup\limits_{m\to \infty} \: (n-1) \frac{ \log \log M_{S_i}(r_m,f)}{\log r_m}\nonumber\\
&\geq \limsup\limits_{m\to \infty} \: (n-1)\frac{\log(\lambda r_m^{q_i})}{\log r_m}\nonumber\\
&= \limsup\limits_{m\to \infty} \: (n-1)\frac{\log \lambda+q_i\log r_m}{\log r_m}\nonumber\\
&= (n-1)q_i \nonumber\\
&= (n-1)d_n m(U_i)^{-1/(n-1)}K^{-2/n}/2. \label{eq:inproof}
\end{align}

Note that 
$$M=\begin{cases}
 \omega_{n-1}, & \text{if }\mu_f \leq K^{-2/n} \left ( \frac{ c_n}{\omega_{n-1} } \right )^{1/(n-1)}\\
 \left((n-1)d_nK^{-2/n}\mu_f^{-1}/2\right)^{n-1}, & \text{if }\mu_f > K^{-2/n} \left ( \frac{ c_n}{\omega_{n-1} } \right )^{1/(n-1)}
\end{cases}.$$

The inequality \eqref{eq:inproof} implies that $$m(U_i) \geq\left((n-1)d_nK^{-2/n}\mu_f^{-1}/2\right)^{n-1}\geq M,$$ which contradicts $m(U_i)<M$ and rules out the second alternative in the Phragm\'en-Lindel\"of result. 

We conclude that $|f(x)|=O(|x|^{p_i'})$ as $|x|\to \infty$ in $S_i$. Since this holds for all $i=1,...,s$ and $\bigcup_{i=1}^s U_i$ is a cover of $S^{n-1}$, we have $|f(x)|=O(|x|^p)$ as $|x| \to \infty$ in $\R^n$ for some $p>0$, which contradicts Lemma \ref{lem:berg}. Hence, there must be a component of $L(f)$, say $E$, with $m(E)\geq \min(c_n K^{(2-2n)/n}\mu_f^{1-n},\omega_{n-1})$.
\end{proof}

\begin{proof}[Proof of Corollary \ref{cor:3}]
    Suppose that every component of $L(f)$ is a point. Let $V_\gamma=\{x_\gamma\}$ be a component of $L(f)$ and $\epsilon>0$. By Lemma \ref{toplemma}, for each component $V_\gamma$, there exists a neighborhood $U_\gamma\subset S^{n-1}$ of $V_\gamma$ with $m(U_\gamma)=\epsilon$ and $\partial U_\gamma\cap L(f)=\emptyset$. From Theorem \ref{sectorLf} in $\R^n\backslash (\cup_\gamma S_{0,U_\gamma})$, we have $|f(x)|=O(|x|^p)$ for some constant $p>0$. In particular, along $\partial S_{0,U_\gamma}$, $|f(x)|=O(|x|^p)$. By Corollary \ref{cor:pl}, in $S_{0,U_\gamma}$ we either have $|f(x)|=O(|x|^{p'})$ for some constant $p'>0$ or \begin{equation}
 \mu_f\geq (n-1)d_n \epsilon^{\frac{-1}{n-1}}K^{-2/n}/2 \label{notnamedyet}
\end{equation} where $d_n>0$ depends only on $n$. In the first case, $f$ is forced to be of polynomial-type by Lemma \ref{lem:berg}, so the order must satisfy \eqref{notnamedyet}. Letting $\epsilon\rightarrow 0$, we have $\mu_f=\infty$. 
\end{proof}

Turning to the inverse problem for quasiregular mappings, the first named author gave a partial answer to the inverse problem in $\R^3$ \cite[Theorem 2.4]{Fletcher}. In \cite{Fletcher}, he considered the case when $E$ consists of the union of the closures of finitely many domains in $S^2$ and showed that there is a quasiregular mapping of finite lower order such that $E=L(f)$. 
This result relied on a construction by Nicks and Sixsmith \cite{NS} which gives a quasiregular mapping of transcendental-type equal to the identity in a half-space. By looking at quasiconformal conjugates of this map, the first named author gave an answer to the inverse problem when $E$ satisfies certain conditions. 

To fully answer the inverse problem, we would first like to construct a quasiregular mapping of transcendental-type with one Julia limiting direction that is equal to the identity outside a half-beam. 
To construct this map, we would use a similar strategy as Nicks and Sixsmith. If we could construct such a map, we could fully answer the inverse problem.

\begin{proof}[Proof of Theorem \ref{inverse}]
    Given a closed set $E\subset S^{n-1}$, we define a quasiregular mapping $f$ inductively as follows.

Let $(\zeta_{1,k})_k$ be a sequence of points in $E$ with $d(\zeta_{1,i},\zeta_{1,j})>1/2$ for $i\neq j$. 
Pick a sequence of points $(x_{1,k})_k$ such that the half-beams $H(x_{1,k},\zeta_{1,k},1)$ are pairwise disjoint. Let $A_{1,k}\colon \R^n\rightarrow\R^n$ be the composition of a translation by $x_{1,k}$ and a rotation that sends $e_1$ to $\zeta_{1,k}$. Note that $A_{1,k}(H(0,e_1,1))=H(x_{1,k},\zeta_{1,k},1)$.
Define $$f_1(x)=\begin{cases}
 A_{1,k}\circ f_0\circ A^{-1}_{1,k}(x), & x\in \cup_k H(x_{1,k},\zeta_{1,k},1)\\
 x, & \text{otherwise}.
\end{cases}$$ Note that this map is quasiregular and $L(f_1)=\{(\zeta_{1,k})_k\}$.

Next, for $m\geq 2,$ consider a sequence of points $(\zeta_{m,k})_k$ in $E$ such that $L(f_{m-1})\subset\{(\zeta_{m,k})_k\}$ and $d(\zeta_{m,i},\zeta_{m,j})>1/2^m$. For each $\zeta_{m,k}\not\in L(f_{m-1})$, pick a point $x_{m,k}$ such that $H(x_{m,k},\zeta_{m,k},1)$ is disjoint from previously defined half-beams $$\bigcup\limits_{h<m}\bigcup\limits_k H(x_{h,k},\zeta_{h,k},1)$$ and $H(x_{m,i},\zeta_{m,i},1)\cap H(x_{m,j},\zeta_{m,j},1)=\emptyset$ when $i\neq j$. Let $A_{m,k}\colon \R^n\rightarrow\R^n$ be the composition of a translation by $x_{m,k}$ and a rotation that sends $e_1$ to $\zeta_{m,k}$. Note that $A_{m,k}(H(0,e_1,1))=H(x_{m,k},\zeta_{m,k},1)$. Define
$$f_m(x)=\begin{cases}
 A_{m,k}\circ f_0\circ A_{m,k}^{-1}(x), & x\in \cup_k H(x_{m,k},\zeta_{m,k},1)\\
 f_{m-1}(x), &\text{otherwise}.
\end{cases}$$
Note that $f_m$ is a quasiregular mapping with $L(f_m)=\{(\zeta_{m,k})_k\}$.

Since $E$ is closed, as $m\rightarrow\infty$, the set of accumulation points of $\{(\zeta_{m,k})\}$ equals $E$. Note that $K(f_m)= K(f_0)$ for $m\geq 1$ since translations and rotations are 1-quasiregular. Moreover, $(f_m)_m$ converges locally uniformly to a map $f$. By \cite[Theorem VI.8.6]{Rickman}, $f$ is a quasiregular map.
Therefore, we have a quasiregular mapping $f$ with $L(f)=E$.

\end{proof}

\bibliographystyle{apalike}
\bibliography{references}

\end{document}